\documentclass{article}%
\usepackage{amsmath}
\usepackage{amsfonts}
\usepackage{amssymb}
\usepackage{graphicx}%
\providecommand{\U}[1]{\protect\rule{.1in}{.1in}}
\newtheorem{theorem}{Theorem}

\newtheorem{corollary}[theorem]{Corollary}

\newtheorem{example}[theorem]{Example}

\newtheorem{lemma}[theorem]{Lemma}

\newtheorem{proposition}[theorem]{Proposition}
\newtheorem{remark}[theorem]{Remark}

\newenvironment{proof}[1][Proof]{\noindent\textbf{#1.} }{\ \rule{0.5em}{0.5em}}
\begin{document}

\begin{center}
{\Large Maximal subalgebras of C*-algebras associated with periodic flows}

\bigskip

Costel Peligrad and L\'{a}szl\'{o} Zsid\'{o}
\end{center}

\bigskip

Costel Peligrad: Department of Mathematical Sciences, University of
Cincinnati, 610A Old Chemistry Building, Cincinnati, OH 45221 USA; E-mail
address: costel.peligrad@uc.edu

\bigskip

L\'{a}szl\'{o} Zsid\'{o}:Dipartimento di Matematica, Universit\`{a} di Roma
"Tor Vergata", Via della Ricerca Scientifica 1, 00133, Roma, Italy; E-mail
address: zsido@mat.uniroma2.it

\bigskip

\textbf{Abstract. }We find necessary and sufficient conditions for the
subalgebra of analytic elements associated with a periodic C*-dynamical system
to be a maximal norm-closed subalgebra. Our conditions are in terms of the
Arveson spectrum of the action. We also describe equivalent properties of the
system in terms of the strong Connes spectrum and the simplicity of the
crossed product.

\bigskip



\section{Introduction}

\bigskip

A major motivation for the study of maximal subalgebras of commutative
C*-algebras stems from an attempt to extend the Stone-Weierstrass
approximation theorem to the case of non-self-adjoint subalgebras. In
particular, Wermer [20] has shown if $A=\mathcal{C}(\mathbf{T),}$ the
C*-algebra of all continuous complex valued functions on $\mathbf{T}=\left\{
z\in\mathbf{C|}\left\vert z\right\vert =1\right\}  $, the closed subalgebra,
$B$, generated by all polynomials in $z$, $B=\left\{  p(z)=\sum_{k=0}^{n}%
a_{k}z^{k}\}|a_{k}\in\mathbf{C,}n\in\mathbf{N}\right\}  $ is a maximal
subalgebra of $A$. Clearly, $B$ is the subalgebra of $A$ consisting of all
continuous functions on $\mathbf{T}$ which can be extended to the unit disk so
as to be analytic in the interior.

Let now $(A,\mathbf{T,}\alpha)$ be a periodic C*-dynamical system and
$A^{\alpha}([0,\infty))$ the subalgebra of analytic elements, i.e. the
subalgebra consisting of all elements of $A$ with non-negative Arveson
spectrum. In this paper we find necessary and sufficient conditions for
$A^{\alpha}([0,\infty))$ to be maximal among all norm-closed subalgebras of
$A$.

The paper is organized as follows. In Section 2 we set the notations that will
be used. In Section 3, we prove that every norm-closed subalgebra that
contains the subalgebra of analytic elements of an one-parameter C*-dynamical
system is globally invariant. In Section 4 we introduce our spectral condition
(S) and prove our main result about the maximality of the subalgebra of
analytic elements of a periodic C*-dynamical system. The condition (S) is
stated in terms of the Arveson spectrum of the action $\alpha$ and is
satisfied if, in particular, the fixed point algebra of the system is a simple
C*-algebra or the Arveson spectrum contains only one positive integer and the
corresponding ideal of the fixed point algebra is simple (Proposition 12). In
the special case when $A$ is the Cuntz C*-algebra $O_{n}$, $n<\infty$, [6],
and $\alpha$ is the gauge action of $\mathbf{T}$ on $A$, the fixed point
algebra of this system is the uniformly hyperfinite C*-algebra $n^{\infty}$
which is a simple C*-algebra and therefore our condition, (S), is satisfied.
We then prove that Condition (S) is equivalent with a deeper property of the
system $(A,\mathbf{T,}\alpha)$ which involves the strong Connes spectrum
defined by Kishimoto [10] and the simplicity of the crossed products. In
Theorem 13, we prove that the condition (S) is equivalent with the maximality
of the subalgebra of analytic elements of the system. Finally, in Proposition
14 we describe the special case of our spectral condition (S) in which the
Arveson spectrum contains only one positive integer and then give examples
when this situation may occur.

In the particular case when $A=C(\mathbf{T})$, the algebra of continuous
functions on $\mathbf{T}=\left\{  z\in\mathbf{C|}\left\vert z\right\vert
=1\right\}  $, and $\alpha$ is the action of $\mathbf{T}$ on $A$ by
translations, the result of J. Wermer that was cited above, follows immediately.

In [[17], Corollary 3.12.], Solel states a necessary and sufficient condition
for the maximality of the subalgebra of analytic elements associated with a
periodic W*-dynamical system $(M,\mathbf{T,}\alpha).$This condition is similar
with our condition (S) and is satisfied if, in particular, the fixed point
algebra, $M^{\alpha}$ is a von Neumann factor. We mention that a von Neumann
factor can be either a simple C$^{\ast}$-algebra (finite factors) or a prime
C*-algebra (infinite factors). Our results show that for a periodic
C*-dynamical system, the maximality of the subalgebra of analytic elements is
related only to the simplicity of some ideals of the fixed point algebra and
not to their primeness.

In [18] it is discussed the maximality of analytic elements of an one
parameter W*-dynamical system, $(M,\mathbf{R},\alpha)$ with two additional conditions:

$M$ is a $\sigma$-finite von Neumann algebra and

$\mathcal{Z}_{M}\cap M^{\alpha}=\mathbf{C1}_{M}$ where $\mathcal{Z}_{M}$ is
the center of $M$ amd $M^{\alpha}$ is the fixed-point algebra of the system.

Our results and methods for periodic C*-dynamical systems and their particular
cases for periodic W*-dynamical systems do not require any of the above conditions.

A related but different direction of studying the subalgebras of analytic
elements, that of subdiagonal algebras, has been initiated in [1]. The study
of subdiagonal algebras has been developed further in [9], for both W* and
C*-dynamical systems and in [8], [4], [15] among others, for W*-dynamical
systems. Obviously, our condition (S) implies that $A^{\alpha}([0,\infty))$
satisfies the less stringent condition of maximality among subdiagonal algebras.

\section{Notations and preliminary results: Spectral subspaces for
one-parmeter dynamical systems}

\bigskip

Let $(X,\mathcal{F)}$ be a dual pair of Banach spaces ([2], [12], [21], [22]).
As in [21], denote by $B_{\mathcal{F}}(X)$ the linear space of all
$\mathcal{F-}$ continuous linear operators on $X$. A one-parameter group
$\left\{  U_{t}\right\}  _{t\in\mathbf{R}}\subset B_{\mathcal{F}}(X)$ is
called $\mathcal{F}-$continuous if for each $x\in X$ and $\varphi
\in\mathcal{F}$, the function $t\mapsto\left\langle U_{t}x,\varphi
\right\rangle $ is continuous. $\left\{  U_{t}\right\}  $ is called bounded if
sup$_{t\in\mathbf{R}}\left\Vert U_{t}\right\Vert <\infty$. Examples of dual
pairs of Banach spaces and one-parameter groups considered in this paper include:

$X=M$, a von Neumann algebra, $\mathcal{F=}M_{\ast}$ its predual and $\left\{
U_{t}\right\}  =\left\{  \alpha_{t}\right\}  $ a one-parameter group of
automorphisms of $M$ such that $t\mapsto\left\langle \alpha_{t}(x),\varphi
\right\rangle $ is continuous for every $x\in X$ and $\varphi\in M_{\ast}$;

$X=A,$ a C*-algebra, $\mathcal{F=}A^{\ast}$ its dual and $\left\{
U_{t}\right\}  =\left\{  \alpha_{t}\right\}  $ a one-parameter group of
automorphisms of $A$ such that $t\mapsto\alpha_{t}(x)\in A$ is continuous for
every $x\in A;$

$X=H$, a Hilbert space and $\left\{  U_{t}\right\}  =\left\{  u_{t}\right\}  $
a strongly continuous one-parameter group of unitary operators on $H$.

An element $\gamma\in\widehat{\mathbf{R}}$ is said to be an essential point
for $U$ if for every neighborhood $V$ of $\gamma,$there is $f\in
L^{1}(\mathbf{R})$ such that $support(\widehat{f})$ is compact and is included
in $V$ and $U_{f}=\int f(t)U_{t}dt\neq0$. Here$\ \widehat{f}$ is the Fourier
transform of $f$. The Arveson spectrum of $U$ is by definition, [3]\newline%
\[
sp(U)=\left\{  \gamma\in\widehat{\mathbf{R}}|\gamma\text{ is an essential
point for }U\right\}
\]

If $sp(U)$ is a non trivial discrete subset of $\mathbf{R}$, then $\left\{
U_{t}\right\}  $ is called periodic and can be viewed as a compact group
$\left\{  U_{g}\right\}  _{g\in G}$, where $G=\mathbf{R/}sp(U)^{\perp}$, where
$sp(U)^{\perp}=\left\{  r\in\mathbf{R|}e^{ir\gamma}=1,\gamma\in sp(U)\right\}
$. The group $G$ can be identified with $\mathbf{T}=\left\{  z\in
\mathbf{C|}\left\vert z\right\vert =1\right\}  $.

Let $x\in X$ and $\gamma\in\widehat{\mathbf{R}}$. Then, [3], $\gamma$ is
called an $U$-essential point for $x$ if for every neighborhood $V$ of
$\gamma,$there is $f\in L^{1}(\mathbf{R})$ such that $support(\widehat{f})$ is
compact and is included in $V$ and $U_{f}(x)=\int f(t)U_{t}xdt\neq0$.
Following [3], we define the Arveson spectrum of $x$,
\[
sp_{U}(x)=\left\{  \gamma\in\widehat{\mathbf{R}})|\gamma\text{ is an
}U\text{-essential point for }x\right\}
\]

If $E\subset\widehat{\mathbf{R}}$ is a closed set, define the spectral
subspace, $[3]$%

\[
X^{U}(E)=\left\{  x\in X|sp_{U}(x)\subset E\right\}
\]

If $O\subset\widehat{\mathbf{R}}$ is an open set and $O=\cup E_{\lambda}$
where $\left\{  E_{\lambda}\right\}  $ is an increasing net of closed sets
such that $O=\cup E_{\lambda}=O=\cup E_{\lambda}^{\circ}$, we will denote
$X^{U}(O)=\overline{\cup X^{U}(E_{\lambda})}$, where the closure is taken in
the $\mathcal{F-}$topology of $X,$ [20]. Note that the notations used in [20]
are slightly different, but clearly defined. We mention that the spectral
subspaces $A^{\alpha}(E)$ can be defined analogously for every C*-dynamical
system $(A,G,\alpha)$ with $G$ a locally compact abelian group [[3], [12]].

If $(A,\mathbf{R},\alpha)$ (respectively $(A,\mathbf{T},\alpha)$) is a
C*-dynamical system, $A^{\alpha}([0,\infty))$ is said to be the subalgebra of
analytic elements of the system.

Let us now recall some concepts from [[21], Section 4]:

We say that $\varphi\in\mathcal{F}$ is absolutely continuous relative to
$\left\{  U_{t}^{\ast}\right\}  $ if $t\mapsto U_{t}^{\ast}\varphi$ is norm
continuous. The set of all absolutely continuous elements of $\mathcal{F}$
relative to $\left\{  U_{t}^{\ast}\right\}  $ is a norm-closed linear subspace
of $\mathcal{F}$ and will be denoted in this paper by $\mathcal{F}%
_{norm}^{U^{\ast}}$.

The following definitions are from [[21], Section 5]. Let $H$ be a Hilbert
space, $S\subset B(H)$ and $K\subset H$ be a closed linear subspace; then we
denote by $[SK]$ the closed linear span of $SK=\left\{  x\xi|x\in S,\xi\in
K\right\}  $.

\begin{remark}
Let $K\subset H$ be a closed subspace and $K_{\lambda}=\cap_{\lambda<\mu
}[M^{\alpha}((-\infty,\mu])K]$, $\lambda\in\mathbf{R}$. Then we have:\newline
i) $K_{\lambda_{1}}\subset K_{\lambda_{2}}$ if $\lambda_{1}\leq\lambda_{2}%
$\newline ii) $K_{\lambda}=\cap_{\lambda<\mu}K_{\mu}$\newline iii) $M^{\alpha
}((-\infty,\nu])K_{\lambda}\subset K_{\nu+\lambda}$\newline and, denoting
$K_{-\infty}=\cap_{\lambda}K_{\lambda}$, $K_{\infty}=\overline{\cup_{\lambda
}K_{\lambda}}$ and $e_{\lambda}$ the orthogonal projection onto $K_{\lambda
},\lambda\in\lbrack-\infty,\infty]$, we have\newline iv) $e_{\infty}$,
$e_{-\infty}\in M^{\prime}$.
\end{remark}

\begin{proof}
i) and ii) are immediate from definitions. iii) is a consequence of the fact
that $M^{\alpha}((-\infty,\nu])M^{\alpha}((-\infty,\lambda])\subset M^{\alpha
}((-\infty,\nu+\lambda])$ [21, Theorem 1.6]. Finally, iv) follows from the
easily checked fact that the two subspaces $e_{\infty}H$ and $e_{-\infty}H$
are invariant for every $x\in M$.
\end{proof}

A closed linear subspace $K\subset H$ is called invariant relative to
$\left\{  \alpha_{t}\right\}  $ if $\cap_{\lambda>0}[M^{\alpha}((-\infty
,\lambda])K]=K$

$K$ is called doubly invariant if $\cap_{\lambda\in\mathbf{R}}[M^{\alpha
}((-\infty,\lambda])K]=K$

$K$ is called simply invariant if $\cap_{\lambda\in\mathbf{R}}[M^{\alpha
}((-\infty,\lambda])K]=\left\{  0\right\}  $.

If $K$ is invariant relative to $\left\{  \alpha_{t}\right\}  $, $e_{\infty}$
is called the support of $K$ [[21], Section 5].

\bigskip

\section{Subalgebras containing the analitic elements, $A^{\alpha}%
([0,\infty))$}

\bigskip

Let $(A,\mathbf{R},\alpha)$ be a one-parameter C*-dynamical system. In this
section we will prove that every norm-closed subalgebra $B\subset A$ such that
$A^{\alpha}([0,\infty))\subset B$ is $\alpha$-invariant.

In [11] it is shown that all $\sigma$-finite von Neumann algebras are
hereditarily reflexive (as stated in the following lemma). Our next result is
the extension of [[11], Corollary 3.7] to the general case of not necessarily
$\sigma$- finite von Neumann algebras.

\begin{lemma}
Let $M\subset B(H)$ be a von Neumann algebra in standard form., $x\in M$,
$S\subset M$ a w*-closed linear subspace such that:\newline%
\[
xS\xi\subset\overline{S\xi},\forall\xi\in H
\]
\newline Then $xS\subset S$. If, in particular, $\mathbf{1}_{M}\in S$, then
$x\in S$.
\end{lemma}

\begin{proof}
Let $\varphi\in M_{\ast}$ be such that $\varphi|_{S}=0$. Then, $\exists
\xi,\eta\in H$ such that $\varphi=\omega_{\xi,\eta}$, that is $\varphi
(z)=<z\xi,\eta>$, $\forall z\in M$, where $<\cdot,\cdot>$ is the scalar
product in $H$ [[19], 5.16 and 10.25]. $\varphi|_{S}=0$ implies that $\eta\bot
S\xi$, hence $\eta\bot\overline{S\xi}$. Since $xS\xi\subset\overline{S\xi}$ it
follows that $\eta\bot xS\xi$. Therefore, $\varphi(xy)=<xy\xi,\eta>=0$
$\forall y\in S$. Applying the Hahn-Banach theorem, it follows that $xy\in S$,
$\forall y\in S$.
\end{proof}

The next Lemma is an extension of [18, Proposition 2.1] to the more general
case of not necessarily $\sigma$-finite von Neumann algebras.

\begin{lemma}
Let $(M,\mathbf{R},\alpha)$ be a W*-dynamical system. Let $N$ be a w*-closed
subalgebra of $M$ containing $M^{\alpha}((-\infty,0))$ and $\mathbf{1}_{M}$.
Then, $N$ is $\alpha$-invariant.
\end{lemma}

\begin{proof}
By Lemma 2, we have to prove that if $M\subset B(H)$ is in standard form,
$x\in N$, $t\in\mathbf{R}$ and $K\subset H$ is a closed linear subspace with
$NK\subset K$, then $\alpha_{t}(x)K\subset K$. Let now $K\subset H$ be a
closed linear subspace such that $NK\subset K$. Then, it is clear that, with
the notations in Remark 1, $K_{0}=\cap_{\mu>0}[M^{\alpha}((-\infty,\mu])K]$ is
invariant relative to $\left\{  \alpha_{t}\right\}  $. By [[21], Theorem 5.1]
$K_{-\infty}$ is a doubly invariant subspace relative to $\left\{  \alpha
_{t}\right\}  $ and $K_{0}\ominus K_{-\infty}$ is simply invariant with
support $e_{\infty}-e_{-\infty}$. By [[21], Theorem 5.2] $[M^{\alpha}%
((-\infty,\lambda])K_{-\infty}]=K_{-\infty}$ $\forall\lambda\in\mathbf{R}$. By
[[21], Theorem 5.3] there exists a strongly continuous one-parameter group of
unitaries $\left\{  u_{t}\right\}  $ on $H,$commuting with $e_{\infty
}-e_{-\infty}$ such that%

\begin{align*}
H^{u}((-\infty,\lambda])  &  =K_{\lambda}\oplus(H\ominus K_{\infty})\text{
and}\\
\alpha_{s}(y)(e_{\infty}-e_{-\infty})  &  =u_{s}yu_{s}^{\ast}(e_{\infty
}-e_{-\infty}),\forall y\in M,\forall s\in\mathbf{R}.
\end{align*}

Let $e$ denote the orthogonal projection on $K$. If $\lambda<0$, then, for
$\lambda<\mu<0$, $M^{\alpha}((-\infty,\mu])K\subset NK\subset K$ by our
assumption on the subspace $K$ and thus $K_{\lambda}\subset K$. If
$\lambda\geq0$, then, for $\lambda<\mu$, $\mathbf{1}_{M}K=K\subset M^{\alpha
}((-\infty,\mu])K$, hence $K_{\lambda}\supset K$. Therefore $e$ commutes with
all $e_{\lambda}$, hence with all orthogonal projections onto spectral
subspaces $H^{u}((-\infty,\lambda])$. It follows that $e$ commutes with every
$u_{s}$, $s\in\mathbf{R}$. Let now $\xi\in K$. Then $\xi=\xi_{1}+\xi_{2}$,
$\xi_{1}\in K_{-\infty}$, $\xi_{2}\in K\ominus K_{-\infty}$ and we have%

\begin{align*}
\alpha_{t}(x)\xi_{1}  &  =\alpha_{t}(x)e_{-\infty}\xi_{1}=e_{-\infty}%
\alpha_{t}(x)\xi_{1}\in K_{-\infty}\subset K\text{ and}\\
\alpha_{t}(x)\xi_{2}  &  =\alpha_{t}(x)(e_{\infty}-e_{-\infty})\xi_{2}%
=u_{t}xu_{t}^{\ast}(e_{\infty}-e_{-\infty})\xi_{2}=\\
&  =u_{t}xu_{t}^{\ast}(e-e_{-\infty})\xi_{2}=u_{t}xu_{t}^{\ast}e\xi
_{2}-e_{-\infty}u_{t}xu_{t}^{\ast}\xi_{2}=\\
&  =eu_{t}xeu_{t}^{\ast}\xi_{2}-e_{-\infty}u_{t}xu_{t}^{\ast}\xi_{2}\in K
\end{align*}

Hence $\alpha_{t}(x)\xi\in K$. Therefore $\alpha_{t}(x)\in K$.
\end{proof}

Next, we will prove an analog of Lemma 3 for C*-dynamical systems. Let $A$ be
a C*-algebra and $\mathcal{F}$ a Banach space in duality with $A$ [[21], page
88]. Let $\alpha$ be an $\mathbf{R}-\sigma(A,\mathcal{F})-$continuous
one-parameter group of $\sigma(A,\mathcal{F})-$continuous automorphisms of $A$.

\begin{lemma}
Let $B\subset A$ be a $\sigma(A,\mathcal{F})-$closed subalgebra containing
$A^{\alpha}((-\infty,0))$. Then $\mathbf{C1}_{\widetilde{A}}+B$ is $\alpha-
$invariant, where $\mathbf{1}_{\widetilde{A}}$ is the unit of the multiplier
algebra of $A$.
\end{lemma}

\begin{proof}
By the discussion following [[21], Corollary 4.2.], we have that
$M=(\mathcal{F}_{norm}^{\alpha^{\ast}})^{\ast}$ is a W*-algebra that,
naturally contains $A$, such that the action $\alpha$ of $\mathbf{R}$ on $A$
extends by w*-continuity to a w*-continuous action of $\mathbf{R}$ on $M$,
still denoted by $\alpha$. Then, by [[21],Theorem 4.1], we have:%

\begin{align*}
\mathcal{F}^{\alpha^{\ast}}((-\infty,\lambda])  &  \subset\mathcal{F}%
_{norm}^{\alpha^{\ast}}\text{ and }\\
\mathcal{F}^{\alpha^{\ast}}([\lambda,\infty))  &  \subset\mathcal{F}%
_{norm}^{\alpha^{\ast}},\forall\lambda\in\mathbf{R}%
\end{align*}

Put $N=\mathbf{C1}_{M}+\overline{B}^{w\ast}$. Then $N$ is a w*-closed
subalgebra of $M$ that contains $\mathbf{1}_{M}$. Let now $\varphi
\in\mathcal{F}_{norm}^{\alpha^{\ast}}$ be such that $\varphi|_{N}=0$. Then
$\varphi(A^{\alpha}((-\infty,0))=\left\{  0\right\}  $ and applying [[21],
Corollary 1.5], it follows that $\varphi\in\mathcal{F}^{\alpha^{\ast}%
}([0,\infty))$. Hence $\varphi\in\mathcal{F}_{norm}^{\alpha^{\ast}}%
([0,\infty))$ and using again [[21], Corollary 1.5], we get $\varphi
|_{M^{\alpha}((-\infty,0))}=0$. Hence, $N$ contains $M^{\alpha}((-\infty,0))$.
By Lemma 3, $N$ is $\alpha$-invariant. Identifying $\widetilde{A}$ with
$\mathbf{C1}_{M}+A,$it follows that $N\cap\widetilde{A}$ is $\alpha
$-invariant. Since, obviously, $N\cap\widetilde{A}\supset\mathbf{C1}_{M}+B,$
in order to prove that $\mathbf{C1}_{M}+B$ is $\alpha$-invariant, it is
sufficient to prove that $N\cap\widetilde{A}\subset\mathbf{C1}_{M}+B$. By the
Hahn-Banach theorem it is enough to prove that if $\psi\in\mathcal{F}$
vanishes on the $\sigma(\widetilde{A},\mathcal{F)}$-closed linear subspace
$\mathbf{C1}_{M}+B$, then it vanishes also on $N\cap\widetilde{A}$. Let
$\psi\in\mathcal{F}$ be such that $\psi|_{\overline{\mathbf{C1}_{M}+B}%
^{\sigma}}=0$. Then, since $A^{\alpha}((-\infty,0))\subset B$, it follows that
$\psi|_{A^{\alpha}((-\infty,0))}=0$. Applying [[21], Corollary 1.5 and Theorem
4.1] we get $\psi\in\mathcal{F}^{\alpha^{\ast}}([0,\infty))\subset
\mathcal{F}_{norm}^{\alpha^{\ast}}=M_{\ast}$. Therefore, $\psi|_{B}=0$ implies
$\psi|_{\overline{B}^{w^{\ast}}}=0$. It follows that $\psi|_{N}=0$ and thus
$\psi|_{N\cap\widetilde{A}}=0$ and we are done.
\end{proof}

\begin{remark}
In the hypotheses of Lemma 4, if $A$ is not unital, then $B$ itself is
$\alpha$-invariant.
\end{remark}

\begin{proof}
Indeed, in this case $N\cap A$ is $\alpha-$invariant and $N\cap A=B$.
\end{proof}

\begin{corollary}
Let $(A,\mathbf{R},\alpha)$ be a one parameter C*-dynamical system. If
$B\subset A$ is a norm-closed subalgebra containing $A^{\alpha}([0,\infty))$,
then $B$ is $\alpha-$invariant$.$
\end{corollary}

\begin{proof}
Since, for every $\lambda\geqq0$, $A^{\alpha}((\lambda,\infty))=A^{\alpha
}((-\infty,\lambda))^{\ast}$ and $sp(\mathbf{1})=\left\{  0\right\}  $ if
$\mathbf{1\in}A,$ the Corollary follows from Lemma 4 and Remark 5.
\end{proof}

$A^{\alpha}([0,\infty))$ will be called the algebra of analytic elements of
the system $(A,\mathbf{R},\alpha)$.

\bigskip

\section{Periodic C*-dynamical systems: The maximality of $A^{\alpha
}([0,\infty))$}

\bigskip

Let now $(C,\mathbf{G},\alpha)$ be a C*-dynamical system with $\mathbf{G}$
compact abelian and $\widehat{\mathbf{G}}$ its dual group of. For $\gamma
\in\widehat{\mathbf{G}}$, denote $C_{\gamma}=C^{\alpha}(\left\{
\gamma\right\}  ).$ Then it is well known and easy to see that $C_{\gamma
}=\left\{  \int_{\mathbf{G}}\overline{<g,\gamma>}\alpha_{g}(x)dg|x\in
C\right\}  =\left\{  x\in C|\alpha_{g}(x)=<g,\gamma>x\right\}  .$Here,
$<g,\gamma>$ denotes the Pontryagin duality. Then, the Arveson spectrum of the
action $\alpha$ is $sp(\alpha)=\left\{  \gamma\in\widehat{\mathbf{G}%
}|C_{\gamma}\neq\left\{  0\right\}  \right\}  $ and $C$ is the closed linear
span of $\cup\left\{  C_{\gamma}|\gamma\in sp(\alpha)\right\}  $ As remarked
above, $C_{-\gamma}=C_{\gamma}^{\ast}$ and thus $-\gamma\in sp(\alpha)$ if
$\gamma\in sp(\alpha)$. For $\gamma=e$, the neutral element of $\widehat
{\mathbf{G}}$, we will denote $C_{e}=C^{\alpha}.$ It is immediate to see that
for every $\gamma\in sp(\alpha)$, $C_{\gamma}C_{-\gamma}=C_{\gamma}C_{\gamma
}^{\ast}=\left\{  \sum_{finite}c_{k}d_{k}^{\ast}|c_{k},d_{k}\in C_{\gamma
}\right\}  $ is a two sided ideal of $C^{\alpha}$. In particular if
$\mathbf{G=T\ }$\ is the set of complex numbers of modulus 1, and
$\alpha:\mathbf{T}\rightarrow Aut\mathbf{(}C)$ an action of $\mathbf{T}$ on
$C$, then, $\widehat{\mathbf{T}}=\mathbf{Z}$ where $\mathbf{Z}$ is the group
of integers. For every $n\in\mathbf{Z}$ denote by $C_{n}=C^{\alpha}(\left\{
n\right\}  )$. In this case the algebra of analytic elements of the system
$(C,\mathbf{T},\alpha),$ $C^{\alpha}([0,\infty))\mathcal{\subset}C$, is the
closed linear span of $\cup\left\{  C_{n}|n\in sp(\alpha),n\geq0\right\}  $.

\begin{remark}
Let $(C,\mathbf{G},\alpha)$ be a C*-dynamical system with $\mathbf{G}$ compact
abelian.Then, the following statements hold:\newline i) every approximate
identity $\left\{  e_{\lambda}\right\}  $ of $C^{\alpha}$ is an approximate
identity of $C$.\newline ii) If $\left\{  e_{\lambda}\right\}  $ is an
approximate identity of the two sided ideal $\overline{C_{\gamma}C_{-\gamma}}%
$, then $\left\{  e_{\lambda}\right\}  $ is a right approximate identity of
$C_{-\gamma}$.\newline iii) $\overline{C_{\gamma}C_{-\gamma}}$ has an
approximate identity $\left\{  e_{\lambda}\right\}  \subset C_{\gamma
}C_{-\gamma}$.
\end{remark}

\begin{proof}
i) Let $\gamma\in sp(\alpha)$ and $c\in C_{\gamma}$Then, since $c^{\ast}c\in
C^{\alpha},$ we have%

\[
\lim_{\lambda}\left\Vert ce_{\lambda}-c\right\Vert ^{2}=\lim_{\lambda
}\left\Vert (ce_{\lambda}-c)^{\ast}(ce_{\lambda}-c)\right\Vert =\lim_{\lambda
}\left\Vert e_{\lambda}c^{\ast}ce_{\lambda}-e_{\lambda}c^{\ast}c-c^{\ast
}ce_{\lambda}+c^{\ast}c\right\Vert =0
\]

Since $C$ is the closed linear span of $\cup\left\{  C_{\gamma}|\gamma\in
sp(\alpha)\right\}  $ we are done.

The proof of the statement ii) is identical. Finally, since $C_{\gamma
}C_{-\gamma}$ is a dense two sided ideal of $\overline{C_{\gamma}C_{-\gamma}}%
$, iii) follows from [[7], Proposition 1.7.2.].
\end{proof}

We will denote by $\mathcal{H}^{\alpha}(C)$ the set of all non zero, $\alpha
-$invariant hereditary subalgebras of $C$. Let $\widetilde{\Gamma}(\alpha)$ be
the strong Connes spectrum of Kishimoto, [10], namely,%

\[
\widetilde{\Gamma}(\alpha)=\left\{  n\in\mathbf{Z}|\overline{D_{n}D_{n}^{\ast
}}=D^{\alpha},\forall D\in\mathcal{H}^{\alpha}(C)\right\}
\]

Then, $\widetilde{\Gamma}(\alpha)$ is a semi group and it plays an important
role in checking the simplicity of the C*-crossed product [10]. The role of
$\widetilde{\Gamma}(\alpha)$ for C*-dynamical systems is similar in many
situations to the one of the Connes spectrum, $\Gamma(\alpha)$ for
W*-dynamical systems.

If $E\subset\widehat{\mathbf{G}}$, denote $E_{\perp}=\left\{  g\in
\mathbf{G|}<g,\gamma>=1,\forall\gamma\in E\right\}  $. If $F\subset G$, denote
$F^{\perp}=\left\{  \gamma\in\widehat{\mathbf{G}}|<g,\gamma>=1,\forall g\in
F\right\}  $.

\begin{lemma}
Let $(C,\mathbf{G},\alpha)$ be a C*-dynamical system with $\mathbf{G}$ compact
abelian. Then the following are equivalent:\newline i) $C^{\alpha}$ is simple
and $sp(\alpha)$ is a subgroup of $\widehat{\mathbf{G}}$\newline ii)
$sp(\alpha)$ is a subgroup of $\widehat{\mathbf{G}}$ and the crossed product
$C\rtimes_{\alpha^{\bullet}}\mathbf{G/}sp(\alpha)_{\perp}$ is a simple
C*-algebra\newline iii) $C$ is $\alpha$-simple and $sp(\alpha)=\widetilde
{\Gamma}(\alpha)$
\end{lemma}

\begin{proof}
Since all three conditions imply that $sp(\alpha)$ is a group, by Pontryagin
duality it follows that $(sp(\alpha)_{\perp})^{\perp}=sp(\alpha)$ in all three
conditions. Replacing the system with $(C,\mathbf{G}/sp(\alpha)_{\perp}%
,\alpha^{\bullet})$, if necessary, we may assume that $\alpha$ is faithful and
$sp(\alpha)=\widehat{\mathbf{G}}$. Here, $\alpha^{\bullet}$ denotes then
quotient action of $\mathbf{G}/sp(\alpha)_{\perp}$ on $C$. By [[14], Corollary
3.8.], condition i) with $sp(\alpha)=\widehat{\mathbf{G}}$ is equivalent with
ii). By [[10] Thm. 3.5.], ii) is equivalent with iii).
\end{proof}

Let $(A,\mathbf{T},\alpha)$ be a C*-dynamical system. In the rest of this
paper we will assume that the action $\alpha$ is non trivial, that is,
$sp(\alpha)\neq\left\{  0\right\}  $. Consider the following spectral property
of the system $(A,\mathbf{T},\alpha)$:

\bigskip

(S) There exists a non zero, closed, two sided ideal $J\subset A^{\alpha}$
which is a simple C*-algebra and such that for every $n\in sp(\alpha)$,
$n\geq1$ $\overline{A_{n}A_{-n}}=J$.

\bigskip

It is obvious that if the fixed point algebra $A^{\alpha}$ is simple, then the
condition (S) is satisfied. In particular, if $A=O_{n}$, $n<\infty$ is the
Cuntz algebra generated by the isometries $\left\{  S_{i}|i=1,2,...n\right\}
$ [6] and $\alpha$ is the gauge action $\alpha_{z}(S_{i})=zS_{i},$ $1\leq
i\leq n$, $z\in\mathbf{T}$, the fixed point algebra, $O^{\alpha}$ is simple
and therefore the condition (S) is satisfied.

We will state and prove conditions that are equivalent to (S) and show that
these equivalent conditions are necessary and sufficient for the algebra of
analytic elements, $A^{\alpha}([0,\infty))$, to be a maximal subalgebra of $A$.

\begin{lemma}
Assume that condition (S) is satisfied. Then, if $n,k\in sp(\alpha)$,
$n,k\geq1$, then $n-k\in sp(\alpha)$.
\end{lemma}

\begin{proof}
From the definition of spectral subspaces, we immediately infer that
$A_{-k}A_{n}\subset A_{n-k}$. If we show that $A_{-k}A_{n}\neq\left\{
0\right\}  $, the conclusion of the lemma follows. Assume to the contrary that
$A_{-k}A_{n}=\left\{  0\right\}  $. Multiplying the previous equality to the
left by $A_{k}$ and to the right by $A_{n}$ we get $A_{k}A_{-k}A_{n}%
A_{-n}=\left\{  0\right\}  $. Hence $JJ=\left\{  0\right\}  $, contradiction
since $JJ=J^{2}=J\neq\left\{  0\right\}  $ by condition (S).
\end{proof}

The following Proposition, describes the Arveson spectrum, $sp(\alpha)$, under
condition (S).

\begin{proposition}
Assume that condition (S) is satisfied. Then either there exists an
$n\in\mathbf{Z}$, $n>0$ such that $sp(\alpha)=\left\{  -n,0,n\right\}  $ or
$sp(\alpha)$ is a subgroup of $\mathbf{Z}$.
\end{proposition}

\begin{proof}
Assume that the first alternative in the statement of the Proposition does not
hold. As $sp(\alpha)\neq\left\{  0\right\}  $, let $n_{0}$ be the smallest
positive element of $sp(\alpha)$. By assumption, there is an $n>n_{0}$, $n\in
sp(\alpha)$. Let $n_{1}$ be the smallest such $n$. By Lemma 9, it follows that
$n_{1}-n_{0}\in sp(\alpha)$. Since $n_{0}$ is the smallest positive element of
$sp(\alpha)$ we have $n_{1}\geq2n_{0}$. Let $k$ be the positive integer such
that $k$ $\geq2$ and $kn_{0}\leq n_{1}<(k+1)n_{0}$. Since $n_{0}$ is the least
positive element of $sp(\alpha)$, applying Lemma 9, it follows that
$n_{1}=kn_{0}$. Another application of Lemma 9 gives $n_{1}-n_{0}%
=(k-1)n_{0}\in sp(\alpha)$. Since $n_{1}=kn_{0}$ is the smallest element in
$sp(\alpha)$ such that $n_{1}>n_{0}$ and $(k-1)n_{0}\in sp(\alpha),$ it
follows that that $k=2$ and thus $n_{1}=2n_{0}$. Condition (S) implies that
$\overline{A_{2n_{0}}A_{-2n_{0}}}=\overline{A_{n_{0}}A_{-n_{0}}}=J$. By
multiplying the previous double eqality by $A_{-n_{0}}$ to the left and by
$A_{n_{0}}$ to the right and using the fact that $A_{n}A_{k}\subset A_{n+k}$
for every $n,k\in sp(\alpha),$ it immediately follows that $\overline
{A_{-n_{0}}A_{n_{0}}}\subset J$ and since $J$ is a simple C*-algebra,
$\overline{A_{-n_{0}}A_{n_{0}}}=J$. By Remark 7 ii), any approximate identity
$\left\{  e_{\lambda}\right\}  $ of $J$ is a right approximate identity of
$A_{n_{0}}$. Hence $A_{n_{0}}=A_{n_{0}}J$ and $\overline{A_{n_{0}}JA_{-n_{0}}%
}=J$. Thus $A_{n_{0}}\overline{A_{2n_{0}}A_{-2n_{0}}}A_{-n_{0}}=J\neq\left\{
0\right\}  $ and so $A_{n_{0}}A_{2n_{0}}\neq\left\{  0\right\}  $. Therefore
$3n_{0}\in sp(\alpha)$. By induction it follows that $sp(\alpha)=\left\{
kn_{0}|k\in\mathbf{Z}\right\}  $ and we are done.
\end{proof}

\begin{remark}
Assume that condition (S) holds and $sp(\alpha)$ is a subgroup of $\mathbf{Z}%
$. Then, $\overline{A_{n}A_{-n}}=J$ for every $n\in sp(\alpha)$, $n\neq0$ (not
only for $n>0$).
\end{remark}

\begin{proof}
Follows from the proof of Proposition 10.
\end{proof}

The following result gives equivalent formulations for the condition (S).

\begin{proposition}
Let $(A,\mathbf{T},\alpha)$ be a C*-dynamical system with $sp(\alpha)\neq(0)$.
The following conditions are equivalent:\newline i) The condition (S)
holds\newline ii) Either \newline ii1) There is a positive integer $n_{0}$
such that $sp(\alpha)=\left\{  -n_{0},0,n_{0}\right\}  $ and $J_{n_{0}%
}=\overline{A_{n_{0}}A_{-n_{0}}}$ is a simple C*-subalgebra of $A^{\alpha}%
$\newline or\newline ii2) There exists an $\alpha$-invariant hereditary
C*-subalgebra, $C\in\mathcal{H}^{\alpha}(A)$, such that the following
conditions hold:\newline a) $A_{n}=C_{n}$ for every $n\in sp(\alpha)$,$n\neq
0$, hence $sp(\alpha)=sp(\alpha|_{C})$\newline b) $C$ is $\alpha
$-simple\newline c) $sp(\alpha|_{C})=\widetilde{\Gamma}(\alpha|_{C})$.
\end{proposition}

\begin{proof}
i) $\Rightarrow$ ii) Assume that the condition (S) holds and ii1) is not
satisfied. Then, Proposition 10 implies that $sp(\alpha)$ is a subgroup of
$\mathbf{Z}$. By Remark 11, $\overline{A_{n}A_{-n}}=J$ for every $n\in
sp(\alpha)$, $n\neq0$. Let $C=\overline{JAJ}$. Then, clearly, $C\in
\mathcal{H}^{\alpha}(A)$ and $C_{n}=A_{n}$ for every $n\in sp(\alpha)$%
,$n\neq0$, so i) $\Rightarrow$ ii2), a). Since $C^{\alpha}=J$ is a simple
C*-algebra, it follows that $C$ is $\alpha$-simple. Hence i) $\Rightarrow$
ii2), b). Since $C^{\alpha}=J$ is simple, and $sp(\alpha|_{C})$ is a subgroup,
applying Lemma 8, it follows that i) $\Rightarrow$ ii2) c) and hence the
implication i) $\Rightarrow$ ii) is proven.

ii) $\Rightarrow$ i) Trivially, ii1)$\Rightarrow$i). Assume now that ii2)
holds. Then, since $sp(\alpha)$ is closed to taking opposites and
$\widetilde{\Gamma}(\alpha)$ is a semigroup [[10], Proposition 2.1.], it
follows that $sp(\alpha|_{C})$ is a group. By Lemma 8, ii2) b) and ii2) c)
imply that $C^{\alpha}$ is simple and $sp(\alpha)$ is a subgroup of
$\mathbf{Z}$. Therefore for every $n\in sp(\alpha)$,$n\neq0$, we have
$\overline{C_{n}C_{-n}}=C^{\alpha}$. Since $A_{n}=C_{n}$ for every such $n$,
i) follows if we set $J=C^{\alpha}\subset A^{\alpha}$.
\end{proof}

We can now state our main result:

\begin{theorem}
Let $(A,\mathbf{T},\alpha)$ be a C*-dynamical system with $\alpha\neq id$.
Then the following conditions are equivalent:\newline i) The condition (S) is
satisfied\newline ii) $A^{\alpha}([0,\infty))$ is a maximal norm-closed
subalgebra of $A$
\end{theorem}

\begin{proof}
i) $\Rightarrow$ ii) Let $B\subseteq A$ be a norm closed subalgebra such that
$A^{\alpha}([0,\infty))\subsetneqq B$. By Corollary 6, $B$ is $\alpha
$-invariant. Since $B\neq A^{\alpha}([0,\infty))$, there exists $n\in
\mathbf{N}$ such that $B_{-n}\neq\left\{  0\right\}  $. Since $A^{\alpha
}([0,\infty))\subset B$ we have $B_{n}=A_{n}$ and $J\subset A^{\alpha}\subset
B$. It follows that $B_{n}B_{-n}$ is a non zero two sided ideal of $J.$ Since
by i), $J$ is simple, it follows that the ideal $B_{n}B_{-n}$ is dense in $J$.
Therefore, by Remark 7 iii), there exists an approximate identity $\left\{
e_{\lambda}\right\}  $ of $J$, $\left\{  e_{\lambda}\right\}  \subset
B_{n}B_{-n}$. From Remark 7 ii), it follows that $\left\{  e_{\lambda
}\right\}  $ is a right approximate identity of $A_{-n}$. Let now $a\in
A_{-n}$ be arbitrary. Since $A_{-n}B_{n}\subset A^{\alpha}\subset B$, we have
that $ae_{\lambda}\in B$ for every $\lambda$. Since $\left\{  e_{\lambda
}\right\}  $ is a right approximate identity of $A_{-n}$, it follows that
$a=\lim ae_{\lambda}\in B$. Hence $B_{-n}=A_{-n}$. According to Proposition
10, either there is an $n\in\mathbf{N}$ such that $sp(\alpha)=\left\{
-n,0,n\right\}  $ or $sp(\alpha)$ is a subgroup of $\mathbf{Z}$. If
$sp(\alpha)=\left\{  -n,0,n\right\}  $, then the above discussion implies that
$B=A$. Assume now that $sp(\alpha)$ is a non zero subgroup of $\mathbf{Z}$, so
there is $n_{0}\in\mathbf{N}$ such that $sp(\alpha)=\left\{  kn_{0}%
|k\in\mathbf{Z}\right\}  $. If $A^{\alpha}([0,\infty))\subsetneqq B\subseteqq
A$ and $n\in\mathbf{N}$ is such that $B_{-n}\neq\left\{  0\right\}  $, then
$n\in\left\{  kn_{0}|k\in\mathbf{Z}\right\}  $. Let $k_{0}\in\mathbf{N}$ be
the smallest natural number such that $B_{-k_{0}n_{0}}\neq\left\{  0\right\}
$. We claim that $k_{0}=1$. Assume that $k_{0}>1$. Then, $B_{-k_{0}n_{0}%
}A_{n_{0}}\neq\left\{  0\right\}  $ since, otherwise $B_{-k_{0}n_{0}}J=(0)$,
so $B_{k_{0}n_{0}}B_{-k_{0}n_{0}}J=\left\{  0\right\}  $. By the above
discussion, $\overline{B_{k_{0}n_{0}}B_{-k_{0}n_{0}}}=J$, and thus $J=(0)$,
contradiction. On the other hand, $\left\{  0\right\}  \neq B_{-k_{0}n_{0}%
}A_{n_{0}}\subset B_{-(k_{0}-1)n_{0}}$. This latter property implies that the
assumption that $k_{0}>1$ is the smallest natural number with $B_{-k_{0}n_{0}%
}\neq\left\{  0\right\}  $ is false, so $k_{0}=1$. Therefore $B_{-n_{0}}%
\neq\left\{  0\right\}  $. An induction argument shows that $B_{-kn_{0}}%
\neq\left\{  0\right\}  $ for every $k\in\mathbf{N}$. By the first part of the
proof, we have that $B_{-n}=A_{-n}$ for every $n\in sp(\alpha)$, $n\neq0$.
Since $A^{\alpha}\subseteqq B$, it follows that $B=A$ and the proof is complete.

ii) $\Rightarrow$ i) Assume that $A^{\alpha}([0,\infty))$ is a maximal norm
closed subalgebra of $A$. Let $n_{0}\in sp(\alpha)$, $n_{0}>0$ be the least
positive element of $sp(\alpha)$. Denote $J_{n_{0}}=\overline{A_{n_{0}%
}A_{-n_{0}}}$. We will prove first that $J_{n_{0}}$ is a simple C*-algebra.
Let $I\subset J_{n_{0}}$ be a non zero, two sided ideal. Consider the folowing
subspace of $A:$%

\[
\mathcal{M=}\overline{lin\left\{  A_{-n}I|n\in sp(\alpha),n>0\right\}
+A^{\alpha}[0,\infty)}%
\]

Then, $\mathcal{M}$ is a norm closed, $\alpha$-invariant subspace of $A$.
Clearly, the set $B=\left\{  a\in A|am\in\mathcal{M},m\in\mathcal{M}\right\}
$ is a norm closed subalgebra of $A$ and $A^{\alpha}([0,\infty))\subset B$.
Since $\left\{  0\right\}  \neq I\subset J_{n_{0}}$, it follows that
$J_{n_{0}}I=I\neq\left\{  0\right\}  $ and therefore there are $a\in
A_{-n_{0}}$ and $i_{0}\in I$ such that $b=ai_{0}\neq0$. Then, since $n_{0}$ is
the least positive element of $sp(\alpha)$, it is clear that $bm\in
\mathcal{M}$ for every $m\in\mathcal{M}$. Hence $b\in B$. Obviously, $b\notin
A^{\alpha}([0,\infty))$ and thus $B$ is a norm closed subalgebra of $A$ such
that $A^{\alpha}([0,\infty))\subsetneqq B$. Since $A^{\alpha}([0,\infty))$ is
by assumption a maximal norm closed subalgebra of $A$, it follows that $B=A$.
Therefore, in particular, $A_{-n_{0}}A^{\alpha}\subset\mathcal{M}$. This means
that $\overline{A_{-n_{0}}A^{\alpha}}\subset\overline{A_{-n_{0}}I}$ which
implies that $J_{n_{0}}\subset I$ and so $I=J_{n_{0}}$. Therefore $J_{n_{0}}$
is a simple C*-algebra as claimed and%
\[
\mathcal{M=}\overline{lin\left\{  A_{-n}J_{n_{0}}|n\in sp(\alpha),n>0\right\}
+A^{\alpha}[0,\infty)}%
\]

Let now $n_{1}\in sp(\alpha)$, $n_{1}\geqq1$ be an arbitrary positive element
of $sp(\alpha)$, so $n_{1}\geq n_{0}$. Since $B=A$, we have $A_{-n_{1}%
}\mathcal{M}\subset\mathcal{M}$, so, in particular $A_{-n_{1}}A^{\alpha
}\subset A_{-n_{1}}J_{n_{0}}$. By multiplying the latter equality to the right
by $A_{n_{1}}$ we get $J_{n_{1}}=\overline{A_{n_{1}}A_{-n_{1}}A^{\alpha}%
}\subset J_{n_{1}}J_{n_{0}}\subset J_{n_{0}}$. Since $J_{n_{0}}$ is simple, it
follows that $J_{n_{1}}=J_{n_{0}}$ and therefore the condition (S) holds.
\end{proof}

In the particular case when $A$ is the crossed product $A=C\ltimes_{\beta
}\mathbf{Z}$, where $C$ is simple, a similar result was obtained in [13].

Next, we will discuss in more detail the structure of a C*-dynamical system
$(A,\mathbf{T},\alpha)$ which satisfies the condition ii1) in Proposition 12.
Let $(A,\mathbf{T},\alpha)$ be a C*-dynamical system with $sp(\alpha)=\left\{
-n_{0},0,n_{0}\right\}  $ for some $n_{0}\in\mathbf{N\subset Z}$ such that the
ideal $I=\overline{A_{n_{0}}A_{-n_{0}}}\subset A^{\alpha}$ is a simple
C*-subalgebra of $A^{\alpha}$ so that the condition ii1) of Proposition 12 is
satisfied. Let $J=\overline{A_{-n_{0}}A_{n_{0}}}$. Then, $A_{n_{0}}$ is an
$I-J$ imprimitivity bimodule in the sense of Rieffel [16] and therefore, $J$
is strongly Morita equivalent with $I$. Hence, $J$ is also a simple
C*-subalgebra of $A^{\alpha}$. Moreover, since $sp(\alpha)$ does not contain
$2n_{0}$ it follows that $IJ=\left\{  0\right\}  $.

Consider the C*-subalgebra $B\subset A$, $B=\overline{(I+J)A(I+J)}$. Then,
obviously, $B$ is an $\alpha-$invariant hereditary C*-subalgebra of $A$,
$B_{n_{0}}=A_{n_{0}}$ and $B^{\alpha}=I+J$. It is easy to show that
$\widetilde{\Gamma}(\alpha)=\left\{  0\right\}  $. The following Proposition
describes the C*-dynamical system in this case:

\begin{proposition}
If $(A,\mathbf{T},\alpha)$ and $B$ are as above, then $B$ is a simple C*-algebra.
\end{proposition}

\begin{proof}
We will prove first that $B$ is $\alpha-$simple, that is if $\left\{
0\right\}  \neq K\subset B$ is a norm-closed $\alpha-$invariant ideal of $B$,
then $K=B$. Let $\left\{  0\right\}  \neq K\subset B$ be such an ideal. Then
$K^{\alpha}\subset B^{\alpha}=I+J$ is a two-sided ideal of $I+J$. Since
$K\neq\left\{  0\right\}  $, it follows that $K^{\alpha}\neq\left\{
0\right\}  $. Therefore, either $K^{\alpha}I\neq\left\{  0\right\}  $, or
$K^{\alpha}J\neq\left\{  0\right\}  $ or both $K^{\alpha}I\neq\left\{
0\right\}  $ and $K^{\alpha}J\neq\left\{  0\right\}  $. In the latter case,
since both $I$ and $J$ are simple C*-algebras, it follows that $K^{\alpha
}=I+J.$ In this case, by applying Remark 7 i) it follows that $A_{n_{0}%
}=A_{n_{0}}J\subset K$ and thus, $K=B$. Next, we will show that the situation
$K^{\alpha}I\neq\left\{  0\right\}  $ and $K^{\alpha}J=\left\{  0\right\}  $
(respectively $K^{\alpha}J\neq\left\{  0\right\}  $ and $K^{\alpha}I=\left\{
0\right\}  $) cannot occur. If $K^{\alpha}I\neq\left\{  0\right\}  $ and
$K^{\alpha}J=\left\{  0\right\}  $ it follows that $K_{n_{0}}=\left\{
0\right\}  =K_{-n_{0}}$. Indeed, if $x\in K_{n_{0}}$, then $x^{\ast}x\in J\cap
K^{\alpha}=\left\{  0\right\}  $, so $x=0$. It then follows that
$sp(\alpha|_{K})=\left\{  0\right\}  $, hence $K=K^{\alpha}=I$. But $I$ is not
an ideal of $B$ since $B_{-n_{0}}=B_{-n_{0}}I\nsubseteqq I$. Therefore, $B$ is
$\alpha-$simple. We prove next that $B$ is a simple C*-algebra. Before
starting the proof of the claim, notice that since $sp(\alpha)$ is, in
particular, a compact subset of $\widehat{\mathbf{T}}=\mathbf{Z}$, by [[12],
Theorem 8.1.12.] the action $\alpha$ is uniformly continuous and therefore,
the dual and the second dual actions $\alpha^{\ast}$ and $\alpha^{\ast\ast}$
are uniformly continuous as well. Therefore, by [[12], Corollary 8.5.3.],
$\alpha^{\ast\ast}|_{\mathcal{Z}}$ is trivial, where $\mathcal{Z}$ is the
center of the second dual $B^{\ast\ast}$ of $B$. Let now $\left\{  0\right\}
\neq K\subset B$ be a norm-closed two-sided ideal of $B$ and $e\in B^{\ast
\ast}$ be the corresponding central projection. Since $\alpha^{\ast\ast
}|_{\mathcal{Z}}$ is trivial, it follows that $\alpha^{\ast\ast}(e)=e$ and
thus $K$ is $\alpha-$invariant. Since $B$ is $\alpha-$simple, it follows that
$K=B$ and the proof is complete.
\end{proof}

\begin{example}
Let $B$ be a simple C*-algebra and $A=M_{2}(B)$, the algebra of 2x2 matrices
with entries in $B$ and $\alpha_{z}\left(  \left[
\begin{array}
[c]{cc}%
a & b\\
c & d
\end{array}
\right]  \right)  =\left[
\begin{array}
[c]{cc}%
a & zb\\
\overline{z}c & d
\end{array}
\right]  $, $z\in\mathbf{T}$. Then, the system $(A,\mathbf{T},\alpha)$
satisfies the hypotheses of the above Proposition.
\end{example}

\begin{proof}
It is easy to show that $\alpha$ is an action of $\mathbf{T}$ on $A$ with
fixed point algebra $A^{\alpha}=\left[
\begin{array}
[c]{cc}%
a & 0\\
0 & d
\end{array}
\right]  $, $A_{1}=\left[
\begin{array}
[c]{cc}%
0 & b\\
0 & 0
\end{array}
\right]  $, $A_{-1}=\left[
\begin{array}
[c]{cc}%
0 & 0\\
c & 0
\end{array}
\right]  $ and $A_{n}=\left[
\begin{array}
[c]{cc}%
0 & 0\\
0 & 0
\end{array}
\right]  $ if $n\notin\left\{  -1,0,1\right\}  .$ Then, if $I=\overline
{A_{1}A_{-1}}=\left[
\begin{array}
[c]{cc}%
a & 0\\
0 & 0
\end{array}
\right]  $, $a\in B$ and $J=\overline{A_{-1}A_{1}}=\left[
\begin{array}
[c]{cc}%
0 & 0\\
0 & d
\end{array}
\right]  $, $d\in B$, we have $IJ=\left\{  0\right\}  $ and the condition (S) holds.
\end{proof}

\begin{example}
More generally, let $I$ and $J$ be two strongly Morita equivalent simple
C*-algebras. Then, by [[5], Theorem 1.1], there is a C*-algebra $B$ such that
$I$ and $J$ are isomorphic with complementary full corners of $B$. The
elements of the algebra $B$ are $2$ x $2$ matrices $\left[
\begin{array}
[c]{cc}%
a & b\\
c & d
\end{array}
\right]  $ where $a\in I$, $d\in J$, and, if $X$ is the imprimitivity bimodule
which establishes the Morita equivalence, $b\in X$ and $c\in\widetilde{X}$
where $\widetilde{X}$ is the dual of $X$ in the sense of Rieffel [[16],
Section 6]. Then if we define $\alpha_{z}\left(  \left[
\begin{array}
[c]{cc}%
a & b\\
c & d
\end{array}
\right]  \right)  =\left[
\begin{array}
[c]{cc}%
a & zb\\
\overline{z}c & d
\end{array}
\right]  $, $z\in\mathbf{T}$, the action $\alpha$ satisfies the hypotheses of
the above proposition with $sp(\alpha)=\left\{  -1,0,1\right\}  $.
\end{example}

\bigskip

\bigskip

{\Large References}

\bigskip

1. W. B. Arveson, Analyticity in operator algebras, Amer. J. Math 89 (1967), 578-642

2. W. B. Arveson, On groups of automorphisms of operator algebras, J. Funct.
Anal. 15 (1974), 217-243.

3. W. B. Arveson, The harmonic analysis of automorphism groups. Operator
algebras and applications, Part I (Kingston, Ont., 1980), pp. 199--269, Proc.
Sympos. Pure Math., 38, Amer. Math. Soc., Providence, R.I., 1982.

4. D. P. Blecher and L. E. Labuschagne, Von Neumann algebraic $H^{p}$ theory,
Function spaces, 89-114, Contemp. Math., 435, Amer. Math. Soc., Providence,
RI, 2007.

5. L. G. Brown, P. Green and M. A. Rieffel, Stable isomorphism and strong
Morita equivalence of C*-algebras, Pacific J. of Math. 71 (1977), 349-363.

6. J. Cuntz, Simple C*-algebras generated by isometries, Commun. Math. Phys.
57 (1977), 173-185.

7. J. Dixmier, Les C*-alg\`{e}bres et leurs repr\'{e}sentations,
Gauthier-Villars, Paris, 1969.

8. R. Exel, Maximal subdiagonal algebras, Amer. J. Math. 110 (1988), 775-782.

9. S. Kawamura and J Tomiyama, On subdiagonal algebras associated with flows
in operator algebras, J. Math. Soc. Japan Volume 29, Number 1 (1977), 73-90.

10. A. Kishimoto, Simple crossed products of C*-algebras by locally compact
abelian groups, Yokohama Math. J. 28 (1980), no. 1-2, 69--85.

11. A. I. Loginov and V. S. Sulman, The hereditary and intermediate
reflexivity for W*-algebras, Izv. Akad. Nauk SSSR, 39 (1975), 1260-1273.

12. G. K. Pedersen, C*-algebras and their automorphism groups, Academic Press,
London New York San Francisco, 1979.

13. C. Peligrad and S. Rubinstein, Maximal subalgebras of C*-crossed products,
Pacific J. Math. 110 (1984), no. 2, 325--333.

14. C. Peligrad, Locally compact group actions on C*-algebras and compact
subgroups, J. Funct. Anal. 76 (1988), 126-139.

15. B. Prunaru, Toeplitz and Hankel operators associated with subdiagonal
algebras, Proc. Amer. Math. Soc. 139 (2011), no. 4, 1387--1396.

16. M. A. Rieffel, Induced representations of C*-algebras, Adv. Math., 13
(1974), 176-257.

17. B. Solel, Algebras of analytic operators associated with a periodic flow
on a von Neumann algebra, Can. J. Math. 37 (1985), 405-429.

18. B. Solel, Maximality of analytic operator algebras, Israel J. Math. 62
(1988), no. 1, 63--89.

19. S. Stratila and L. Zsido, Lectures on von Neumann algebras, Abacus Press, 1975.

20. J. Wermer, On algebras of continuous functions, Proc. Amer. Math. Soc., 4
(1953), 866-869.

21. L. Zsido, Spectral and ergodic properties of the analytic generators, J.
Approximation Theory 20 (1977), 77--138.

22. L. Zsido, On spectral subspaces associated to locally compact abelian
groups of operators, Advances in Math. 36 (1980), 213-276.

\end{document}